\documentclass{article}%
\usepackage{amsmath}
\usepackage{amsfonts}
\usepackage{amssymb}
\usepackage{graphicx}%
\providecommand{\U}[1]{\protect\rule{.1in}{.1in}}
\newtheorem{theorem}{Theorem}[section]

\newtheorem{corollary}[theorem]{Corollary}

\newtheorem{definition}[theorem]{Definition}
\newtheorem{example}[theorem]{Example}

\newtheorem{lemma}[theorem]{Lemma}

\newtheorem{proposition}[theorem]{Proposition}
\newtheorem{remark}[theorem]{Remark}

\newenvironment{proof}[1][Proof]{\noindent\textbf{#1.} }{\ \rule{0.5em}{0.5em}}
\begin{document}

\title{ A note on the Dixmier-Moeglin Equivalence in Leavitt Path Algebras of
arbitrary graphs over a field}
\author{Kulumani M. Rangaswamy\\Department of Mathematics, University of Colorado, \\Colorado Springs, Colorado 80918, USA}
\date{}
\maketitle

\begin{abstract}
The Dixmier-Moeglin Equivalence (for short, the DM-equivalence) is the
equivalence of three distinguishing properties of prime ideals in a
non-commutative algebra A. These properties are of (i) being primitive, (ii)
being rational and (iii) being locally closed in the Zariski \ topology of
Spec(A). The DM-equivalence holds in many interesting algebras over a field.
Recently, it was shown that the prime ideals of a Leavitt path algebra of a
finite graph satisfy the DM-equivalence, In this note, we investigate the
occurrence of the DM-equivalence in a Leavitt path algebra L of\ an arbitrary
directed graph E. Our analysis shows that locally closed prime ideals of L
satisfy interesting equivalent properties such as being strongly primitive and
completely irreducible. Examples illustrate the results and the constraints.

\end{abstract}

\section{\bigskip Introduction}

In the study of several classes of non-commutative algebras, an important but
often difficult question is to classify the irreducible representations of an
algebra and this question is usually reduced to characterizing the primitive
ideals which are the kernels of these representations. Towards this goal, it
is often necessary to identify the primitive ideals among the prime ideals of
these algebras. A fundamental theorem in this connection was proved by J.
Dixmier \cite{D} and C, Moeglin \cite{MO}: If $U$ is the enveloping algebra of
a finite dimensional complex Lie algebra, and $P$ is a prime ideal of $U$,
then $P$ is primitive if and only if it is rational, if and only if $P$ is a
locally closed point in the prime spectrum of $U$ under the Zariski topology.
Here, a prime ideal $P$ of an algebra $A$ over a field $K$ is said to be
\textit{rational} if the center of the Martindale ring of quotients of prime
algebra $A/P$ -- called the extended centroid of $A/P$ -- is an algebraic
extension of $K$. Likewise, an equivalent statement for a prime ideal $P$ to
be \textit{locally closed} in the algebra $A$ is that the intersection $I$ of
all the prime ideals of $A$ properly containing $P$ strictly contains $P$. An
algebra $A$ over a field $K$ is said to satisfy the \textit{Dixmier-Moeglin
Equivalence }(for short, DM-equivalence) on prime ideals, if for any prime
ideal $P$ of $A$, the three properties in the Dixmier-Moeglin theorem are
equivalent. The D-M equivalence seems to hold in a number of interesting
classes of algebras. Just to name a few, Goodearl and Letzter \cite{GL}
established the DM-equivalence in Quantum coordinate rings and also in
quantized Weyl algebras, while Bell, Rogalski and Sierra (\cite{BRS}) proved
the DM-equivalence for prime ideals in the Twisted homogeneous coordinate rings.

Inspired by the above investigations, it was shown in \cite{ABR-1} that if $E$
is a finite graph and $K$ is a field, then the DM-equivalence for prime ideals
always holds in the Leavitt path algebra $L_{K}(E)$. The aim of this note is
to investigate the occurrence of the DM-equivalence in the case of a Leavitt
path algebra $L:=L_{K}(E)$ of an arbitrary graph $E$. We begin by obtaining
several characterizations of a locally closed prime ideal $P$ of $L$. Such an
ideal $P$, in addition to being primitive, has the interesting property of
being Completely irreducible, a type of ideals that play an important role in
the multiplicative ideal theory of non-noetherian commutative rings (see
definition below and (\cite{FHO}). This leads to $P$ possessing a property
stronger than being primitive which we call strong primitivity. \ It is shown
that, in the case of a finite graph $E$, the primitive ideals of $L_{K}(E)$
are always strongly primitive, but this does not hold if $E$ is an infinite
graph. We also establish that a locally closed prime ideal $P$ of $L$ is
always rational. The converse holds if $P$ is a non-graded prime ideal or a
graded prime ideal of the form $P=I(H,B_{H}\backslash\{u\})$. But an example
of a countably infinite graph $E$ shows that the corresponding Leavitt path
algebra $L_{%
\mathbb{C}
}(E)$ over the field $%
\mathbb{C}
$ of complex numbers possesses many graded prime ideals which are rational but
not locally closed.

\section{Preliminaries}

For the general notation, terminology and results in Leavitt path algebras, we
refer to \cite{AAS}. We give below a short outline of some of the needed basic
concepts and results.

A (directed) graph $E=(E^{0},E^{1},r,s)$ consists of two sets $E^{0}$ and
$E^{1}$ together with maps $r,s:E^{1}\rightarrow E^{0}$. The elements of
$E^{0}$ are called \textit{vertices} and the elements of $E^{1}$
\textit{edges}. All the graphs $E$ that we consider (excepting when
specifically stated) are arbitrary in the sense that no restriction is placed
either on the number of vertices in $E$ or on the number of edges emitted by a
single vertex.

A vertex $v$ is called a \textit{sink} if it emits no edges and a vertex $v$
is called a \textit{regular} \textit{vertex} if it emits a non-empty finite
set of edges. An \textit{infinite emitter} is a vertex which emits infinitely
many edges. A graph without infinite emitters is said to be
\textit{row-finite}. For each $e\in E^{1}$, we call $e^{\ast}$ a \textit{ghost
edge}. We let $r(e^{\ast})$ denote $s(e)$, and we let $s(e^{\ast})$ denote
$r(e)$. A\textit{\ path} $\mu$ of length $n>0$ is a finite sequence of edges
$\mu=e_{1}e_{2}\cdot\cdot\cdot e_{n}$ with $r(e_{i})=s(e_{i+1})$ for all
$i=1,\cdot\cdot\cdot,n-1$. In this case $\mu^{\ast}=e_{n}^{\ast}\cdot
\cdot\cdot e_{2}^{\ast}e_{1}^{\ast}$ is the corresponding \textit{ghost path}.
A vertex is considered a path of length $0$. The set of all vertices on the
path $\mu$ is denoted by $\mu^{0}$.

A path $\mu$ $=e\cdot\cdot\cdot e_{n}$ in $E$ is \textit{closed} if
$r(e_{n})=s(e_{1})$, in which case $\mu$ is said to be based at the vertex
$s(e_{1})$. A closed path $\mu$ as above is called \textit{simple} provided it
does not pass through its base more than once, i.e., $s(e_{i})\neq s(e_{1})$
for all $i=2,...,n$. The closed path $\mu$ is called a \textit{cycle} if it
does not pass through any of its vertices twice, that is, if $s(e_{i})\neq
s(e_{j})$ for every $i\neq j$. An \textit{exit} for a path $\mu$ $=e_{1}%
\cdot\cdot\cdot e_{n}$ is an edge $f$ that satisfies $s(f)=s(e_{i})$ for some
$i$ but $f\neq e_{i}$. The graph $E$ is said to satisfy \textit{Condition
(L)}, if every cycle in $E$ has an exit in $E$. The graph $E$ is said to
satisfy\textit{\ Condition (K),} if any vertex on a closed path $\mu$ is also
the base for a closed path $\gamma$ different from $\mu$.

If there is a path from vertex $u$ to a vertex $v$, we write $u\geq v$. A
non-empty subset $D$ of vertices is said to be \textit{downward directed }\ if
for any $u,v\in D$, there exists a $w\in D$ such that $u\geq w$ and $v\geq w$.
A subset $H$ of $E^{0}$ is called \textit{hereditary} if, whenever $v\in H$
and $w\in E^{0}$ satisfy $v\geq w$, then $w\in H$. A hereditary set is
\textit{saturated} if, for any regular vertex $v$ in $E^{0}$, $r(s^{-1}%
(v))\subseteq H$ implies $v\in H$.

\begin{definition}
Given an arbitrary graph $E$ and a field $K$, the \textit{Leavitt path algebra
}$L_{K}(E)$ is defined to be the $K$-algebra generated by a set $\{v:v\in
E^{0}\}$ of pair-wise orthogonal idempotents together with a set of variables
$\{e,e^{\ast}:e\in E^{1}\}$ which satisfy the following conditions:

(1) \ $s(e)e=e=er(e)$ for all $e\in E^{1}$.

(2) $r(e)e^{\ast}=e^{\ast}=e^{\ast}s(e)$\ for all $e\in E^{1}$.

(3) (The "CK-1 relations") For all $e,f\in E^{1}$, $e^{\ast}e=r(e)$ and
$e^{\ast}f=0$ if $e\neq f$.

(4) (The "CK-2 relations") For every regular vertex $v\in E^{0}$,
\[
v=\sum_{e\in E^{1},s(e)=v}ee^{\ast}.
\]

\end{definition}

Every Leavitt path algebra $L_{K}(E)$ is a $%
\mathbb{Z}
$\textit{-graded algebra}, namely, $L_{K}(E)=%
{\displaystyle\bigoplus\limits_{n\in\mathbb{Z}}}
L_{n}$ induced by defining, for all $v\in E^{0}$ and $e\in E^{1}$, $\deg
(v)=0$, $\deg(e)=1$, $\deg(e^{\ast})=-1$. Here, for each $n\in%
\mathbb{Z}
$, the homogeneous component $L_{n}$ is given by
\[
L_{n}=\{%
{\textstyle\sum}
k_{i}\alpha_{i}\beta_{i}^{\ast}\in L:\text{ }|\alpha_{i}|-|\beta_{i}|=n\}.
\]
Further, the $L_{n}$ are abelian subgroups satisfying $L_{m}L_{n}\subseteq
L_{m+n}$ for all $m,n\in%
\mathbb{Z}
$. An ideal $I$ of $L_{K}(E)$ is said to be a \textit{graded ideal} if $I=$ $%
{\displaystyle\bigoplus\limits_{n\in\mathbb{Z}}}
(I\cap L_{n})$. Equivalently, if $a\in I$ and $a=a_{i_{1}}+\cdot\cdot
\cdot+a_{i_{m}}$ is a graded sum with $a_{i_{k}}\in L_{i_{k}}$ for all
$k=1,\cdot\cdot\cdot,m$, then $a_{i_{k}}\in I$ for all $k$.

\bigskip

A \textit{breaking vertex }of a hereditary saturated subset $H$ is an infinite
emitter $w\in E^{0}\backslash H$ with the property that $0<|s^{-1}(w)\cap
r^{-1}(E^{0}\backslash H)|<\infty$. The set of all breaking vertices of $H$ is
denoted by $B_{H}$. For any $v\in B_{H}$, $v^{H}$ denotes the element
$v-\sum_{s(e)=v,r(e)\notin H}ee^{\ast}$. Given a hereditary saturated subset
$H$ and a subset $S\subseteq B_{H}$, $(H,S)$ is called an \textit{admissible
pair.} Given an admissible pair $(H,S)$, the ideal generated by $H\cup
\{v^{H}:v\in S\}$ is denoted by $I(H,S)$. It was shown in \cite{T} that the
graded ideals of $L_{K}(E)$ are precisely the ideals of the form $I(H,S)$ for
some admissible pair $(H,S)$. Moreover, $L_{K}(E)/I(H,S)\cong L_{K}%
(E\backslash(H,S))$. Here $E\backslash(H,S)$ is a \textit{quotient graph of
}$E$ where\textit{\ }%
\[
(E\backslash(H,S))^{0}=(E^{0}\backslash H)\cup\{v^{\prime}:v\in B_{H}%
\backslash S\}
\]
and
\[
(E\backslash(H,S))^{1}=\{e\in E^{1}:r(e)\notin H\}\cup\{e^{\prime}:e\in
E^{1}\text{ with }r(e)\in B_{H}\backslash S\}
\]
and $r,s$ are extended to $(E\backslash(H,S))^{0}$ by setting $s(e^{\prime
})=s(e)$ and $r(e^{\prime})=r(e)^{\prime}$.

We will also be using the fact that the Jacobson radical (and in particular,
the prime/Baer radical) of $L_{K}(E)$ is always zero (see \cite{AAS}).

The following statement will be useful in our investigation.

\begin{theorem}
\label{Prime ideals of LPA}(Theorem 3.12, \cite{R-1}) Let $P$ be an ideal of
$L_{K}(E)$ with $I\cap E^{0}=H$. Then $P$ is a prime ideal if and only if $P$
satisfies one of the following conditions:
\end{theorem}

\begin{theorem}
(i) \ $\ P=I(H,B_{H})$ such that $E^{0}\backslash H$ is downward directed;

(ii) $\ P=I(H,B_{H}\backslash\{u\})$ where $u\in B_{H}$ and $v\geq u$ for all
$v\in E^{0}\backslash H$;

(iii) $P=I(H,B_{H})+<p(c)>$ where $c$ is a cycle without exits in
$E^{0}\backslash H$, $E^{0}\backslash H$ is downward directed and $p(x)$ is an
irreducible polynomial in $K[x,x^{-1}]$.
\end{theorem}

\begin{definition}
Let $K$ be a field and let $A$ be a prime $K$-algebra. The Martindale ring of
quotients of $A$, denoted by $Frac(A)$ consists of equivalent classes of pairs
$(I,f)$ where $I$ is a non-zero ideal of $A$ and $f\in Hom_{A}(I_{A},A_{A})$.
Here two pairs $(I,f),(J,g)$ are considered equivalent if $f=g$ on the
intersection $I\cap J$. Addition and multiplication is given by
\[
(I,f)+(J,g)=(I\cap J,f+g)\text{ and }(I,f)\cdot(J,g)=(IJ,f\circ g)
\]
The center of $Frac(A)$ is called the \textit{extended centroid} of $A$ and is
denoted by $C(A)$. For each $a\in A$, let $\lambda_{a}$ be the map
$A\longrightarrow A$ given by $\lambda_{a}(x)=ax$. Then the map
$a\longrightarrow\lambda_{a}$ is an embedding of $A$ in $Frac(A)$.
\end{definition}

We will be using an internal characterization the extended centroid $C(A)$
given in\ \cite{BM}: $C(A)$ consists precisely of the equivalence classes
$[(I,f)]$, where $f:I\longrightarrow A$ is a $(A,A)$ bimodule homomorphism.
Thus%
\[
C(A)=\{(I,f):0\neq I\vartriangleleft A,f\in Hom(_{A}I_{A},_{A}A_{A})\}/\sim
\]
where $\sim$ is the equivalence relation described above.

\section{The Results}

In this section, we shall investigate the occurrence of the D-M equivalence
for prime ideals in a Leavitt path algebra $L=:L_{K}(E)$ of an arbitrary
directed graph $E$ over a field $K$. \ First, we obtain several
characterizations of a locally closed prime ideal $P$ of $L$. In particular,
such a prime ideal $P$ actually satisfies a property stronger than
primitivity. We call ideals with this property strongly primitive ideals.
Strongly primitive ideals in $L_{K}(E)$ have interesting property of being
completely irreducible. We show that, if $E$ is an arbitrary graph, a prime
ideal $P$ of $L_{K}(E)$ is locally closed if and only if $P$ is strongly
primitive. We also show that a strongly primitive ideal $P$ of $L_{K}(E)$ is
rational. The converse easily holds if the prime ideal $P$ is either a
non-graded ideal or a graded ideal of the form $P=I(H,B_{H}\backslash\{u\})$
where $H=P\cap E^{0}$ and $B_{H}$ is the set of breaking vertices
corresponding to $H$. However, by means of two examples, we show that a graded
rational prime ideal of $L_{K}(E)$ need not be strongly primitive and thus not
locally closed.

We begin by stating the following description of the primitive ideals of a
Leavitt path algebra. $L=:L_{K}(E)$ of an arbitrary graph $E$ over a field $K$.

\begin{theorem}
\label{Primitive ideals of LPAs} (Theorem 4.3, \cite{R-1}) Let $E$ be an
arbitrary graph and let $L:=L_{K}(E)$. An ideal $P$ of $L$ with graded part
$gr(P)=I(H,S)$ is primitive if and only if $P$ satisfies one of the following:

(a) $P$ is a non-graded prime ideal of the form $P=I(H,B_{H})+<p(c)>$where
$p(x)\in K(x)$\ is an irreducible polynomial and $c$ is a cycle without exits
in $E^{0}\backslash H$;

(b) $P$ is a graded prime ideal of the form $P=I(H,B_{H}\backslash\{u\}))$ for
some $u\in B_{H}$;

(c) $P$ is a graded prime ideal of the form $P=I(H,B_{H})$ and $E\backslash
(H,B_{H})^{0}=E^{0}\backslash H$ satisfies both Condition (L) and the CSP.
\end{theorem}

In statement (c) above, "CSP"\ means the countable separation property which
is defined below.

\begin{definition}
A graph $E$ is said to satisfy the countable separation property (for short,
the \textbf{CSP}), if there is a countable subset $S$ of $E^{0}$ such that,
for every $u\in E^{0}$ there is a $v\in S$ such that $u\geq v$, that is, there
is a path from $u$ to $v$.
\end{definition}

Influenced by Theorem \ref{Primitive ideals of LPAs}, we introduce the
following definition.

\begin{definition}
(i) A graph $E$ is said to satisfy the \textbf{strong CSP}, if $E^{0}$
satisfies the CSP with respect to a countable set $S$ such that $S$ is
contained in every non-empty hereditary saturated subset of $E^{0}$;

(ii) Given an arbitrary graph $E$, an ideal $P$ of $L_{K}(E)$ with graded part
$gr(P)=I(H,S)$ is said to be \textbf{strongly primitive} if the quotient graph
$E\backslash(H,S)$ is downward directed and satisfies both Condition (L) and
the strong CSP.
\end{definition}

\begin{remark}
\label{Str. primitive} Suppose (a) $P$ is a non-graded prime ideal of $L$, say
$P=I(H,B_{H})+<p(c)>$ with $c$ a no exit cycle in $E\backslash(H,B_{H})$ or
(b) $P$ is a graded prime ideal of the form $P=I(H,B_{H}\backslash\{u\})$.
Then $P$ is strongly primitive in both cases. To see this, note that in case
(a), by Theorem \ref{Primitive ideals of LPAs}, the quotient graph
$E\backslash(H,B_{H})$ is downward directed satisfies the strong CSP with
respect to $\{c^{0}\}$ and likewise in case (b), the quotient graph
$E\backslash(H,B_{H}\backslash\{u\})$ satisfies the strong CSP with respect to
the sink $u^{\prime}$ corresponding to $u$.
\end{remark}

\begin{definition}
(i) \ \ A prime ideal $P$ in a ring $R$ is said to be \textbf{locally closed}
if $P$ is not the intersection of all the prime ideals of $R$ properly
containing $P$.

(ii) A prime ideal $P$ in an algebra $R$ over a field $K$ is said to be
\textbf{rational} if the extended centroid of $R/P$ is an algebraic extension
of \ the field $K$.

(iii) (\cite{FHO}) \ An ideal $I$ in a ring $R$ is said to be
\textbf{completely irreducible} if $I$ is not the intersection of all the
ideals of $R$ properly containing $I$.
\end{definition}

We wish to show that in a Leavitt path algebra $L$, a prime ideal $P$ is
locally closed\
$<$
=
$>$
$P$ is completely irreducible
$<$
=
$>$
$P$ is strongly primitive. We shall be using the following three results.

\begin{lemma}
\label{Zorn ring} (Theorem 2.1, \cite{R}) An arbitrary graph $E$ satisfies
Condition (L) if and only if $L:=L_{K}(E)$ is a Zorn ring, that is non-zero
left/right ideal of $L$ contains a non-zero idempotent.
\end{lemma}

\begin{lemma}
\label{No vertices} (Lemma 3.5, \cite{R-1}) Let $E$ be a downward directed
graph and let $L:=L_{K}(E)$. If $A$ is a non-zero ideal of $L$ containing no
vertices, then there is a cycle $c$ without exits in $E$ and $A=<f(c)>$, where
$f(x)\in K[x]$ with a non-zero constant term.\bigskip
\end{lemma}

\begin{theorem}
\label{Compl. irreducible ideals} (Theorem 5.5, \cite{MR}) Let $E$ be an
arbitrary graph and $L:=L_{K}(E)$. A proper ideal $I$ of $L$ is completely
irreducible if and only if one of the following conditions hold:

(a) $I=I(H,S)$ is a graded ideal which is strongly primitive;

(b) $I=P^{n}$ for some non-graded prime ideal $P$ and positive integer $n$.
\end{theorem}

\textbf{Note}: Since $A^{2}=A$ for any graded ideal of $L$, Remark
\ref{Str. primitive} implies that Theorem \ref{Compl. irreducible ideals} can
be rephrased as follows: An ideal $I$ \ of $L$ is completely irreducible if
and only if $I=P^{n}$ for some strongly primitive ideal $P$ and positive
integer $n$.

The next theorem gives a characterization of the locally closed prime ideals
in a Leavitt path algebra.

\begin{theorem}
\label{loc. closed <=>Strong. primitive} Let $E$ be an arbitrary graph and
$L:=L_{K}(E)$. Then the following properties are equivalent for a prime ideal
$P$ of $L$:

(a) $P$ is locally closed;

(b) $P\ $is completely irreducible;

(c) There is an element $a\in L$ such that $P$ is maximal with respect to the
property that $a\notin P$;

(d) \ $P\ $is strongly primitive.
\end{theorem}

\begin{proof}
(a) =
$>$
(b). Suppose $P$ is a locally closed prime ideal of $L$. If $P$ is a
non-graded prime ideal or is a graded prime ideal of the form $P=I(H,B_{H}%
\backslash\{u\})$, then $P$ is completely irreducible by Theorem
\ref{Compl. irreducible ideals} and Remark \ref{Str. primitive}. So assume
that $P$ is a graded prime ideal of the form $P=I(H,B_{H})$. Since the
Jacobson radical of $L/P\cong L_{K}(E\backslash(H,B_{H}))$ is zero, the
intersection of all the primitive ideals in $L/P$ is $\{0\}$. As $\{0\}$ is
locally closed in $L/P$, we conclude that $\{0\}$ is necessarily a primitive
ideal of $L/P\cong L_{K}(E\backslash(H,B_{H}))$. In particular, $E\backslash
(H,B_{H})$ is downward directed and satisfies Condition (L). Let $I\neq\{0\}$
be the intersection of all the non-zero prime ideals of $L/P\cong
L_{K}(E\backslash(H,B_{H}))$. Since $E$ satisfies Condition (L), Lemma
\ref{Zorn ring} implies that $I$ contains a non-zero idempotent, say $u$. We
claim that every non-zero ideal of $L_{K}(E\backslash(H,B_{H}))$ contains $u$.
\ Because, if there is a non-zero ideal $A$ of $L_{K}(E\backslash(H,B_{H}))$
which does not contain $u$, then Zorn's Lemma will give rise to a non-zero
ideal $P$ of $L_{K}(E\backslash(H,B_{H}))$ maximal with respect to the
property that $u\notin P$. We claim that this ideal $P$ is a prime ideal. To
see this, suppose $a,b$ are elements of $L$ such that $a\notin P$ and $b\notin
P$. Since $u$ belongs to $P+RaR$ and $P+RbR$, we have%
\[
u=u^{2}\in(P+RaR)(P+RbR)=P^{2}+PRbR+RaRP+RaRbR\subseteqq P+R(aRb)R.
\]
As $u\notin P$, $aRb\nsubseteqq P$. Thus $P$ is a prime ideal and so
$P\supseteqq I\backepsilon u$, a contradiction. Consequently, the idempotent
$u$ belongs to every non-zero ideal of $L_{K}(E\backslash(H,B_{H}))$. This
shows that $\{0\}$ is completely irreducible in $L_{K}(E\backslash
(H,B_{H}))\cong L/P$ and hence $P$ is completely irreducible\ in $L$.

(b) =
$>$
(a) \ Because, a completely irreducible prime ideal is, in particular, locally
closed by definition.

Now (b)
$<$
=
$>$
(c) easily follows from the definition of complete irreducibility and (d)
$<$
=
$>$
(b) from Theorem \ref{Compl. irreducible ideals} \ and Remark
\ref{Str. primitive}.
\end{proof}

\begin{remark}
From the proof of (a) =
$>$
(b) in the above theorem, it is clear that if a graded ideal $I$ of a Leavitt
path algebra is not the intersection of all the prime ideals properly
containing $I$, then $I$ itself must be a prime ideal.
\end{remark}

The following example shows that, in general, a primitive ideal of $L_{K}(E)$
need not be strongly primitive.

\begin{example}
\label{Primitive need not be str. primitive}Consider the following
\textquotedblleft$%
\mathbb{N}
\times%
\mathbb{N}
$-Lattice\textquotedblright\ graph $E$ where the vertices in $E$ are points in
the first quadrant of the coordinate plane whose coordinates are integers
$\geq0$. Specifically, $E^{0}=\{(m,n):m,n\in%
\mathbb{Z}
$ with $m,n\geq0\}$. Every vertex $(m,n)$ emits two edges connecting $(m,n)$
with $(m+1,n)$ and $(m,n+1)$.%

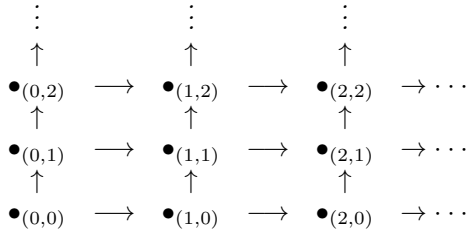
\[%
\begin{array}
[c]{cccccc}%
\vdots &  & \vdots &  & \vdots & \\
\uparrow &  & \uparrow &  & \uparrow & \\
\bullet_{(0,2)} & \longrightarrow & \bullet_{(1,2)} & \longrightarrow &
\bullet_{(2,2)} & \rightarrow\cdot\cdot\cdot\\
\uparrow &  & \uparrow &  & \uparrow & \\
\bullet_{(0,1)} & \longrightarrow & \bullet_{(1,1)} & \longrightarrow &
\bullet_{(2,1)} & \rightarrow\cdot\cdot\cdot\\
\uparrow &  & \uparrow &  & \uparrow & \\
\bullet_{(0,0)} & \longrightarrow & \bullet_{(1,0)} & \longrightarrow &
\bullet_{(2,0)} & \rightarrow\cdot\cdot\cdot
\end{array}
\]

Now $E^{0}$ is downward directed and $E$ contains no cycles and hence
trivially satisfies Condition (L). Since $E^{0}$ is countable, $E^{0}$
satisfies the CSP (for example, with respect to set $\{(m.n):m\geq1,n\geq1\}$.
Thus the ideal $\{0\}$ is a primitive ideal. But $\{0\}$ is not strongly
primitive. To justify this, it is enough if we could find a set of hereditary
saturated subsets of $E^{0}$ whose intersection is the empty set. Indeed, for
each $n\geq1$, the sets $H_{n}=\{(i,j):i\geq0,j\geq n\}$ are hereditary
saturated subsets of $E^{0}$ whose intersection is the empty set.
\end{example}

In contrast to Example \ref{Primitive need not be str. primitive}, we show
that, if $E$ is a finite graph, then the primitive ideals\ of $L_{K}(E)$ are
always strongly primitive. In the proof, we will use the following Definition
and Lemma \ref{Lemma 6}.

\begin{definition}
( \cite{AAS}) A cycle $c$ in a graph $E$ is said to be an extreme cycle if $c$
has exits and, for every path $\alpha$ starting at a vertex in $c^{0}$, there
is a path $\beta$ such that $s(\beta)=r(\alpha)$ and $r(\alpha\beta)\in c^{0}%
$. Intuitively, every path that leaves the cycle $c$ can be extended to a
longer path that ends at a vertex on $c$.
\end{definition}

\begin{lemma}
\label{Lemma 6}(\cite{AAS}, Lemma 3.7.10) Let $E$ be a finite graph. Then
every vertex in $E$ connects to a sink or a no exit cycle or an extreme cycle.
\end{lemma}

\begin{proposition}
Let $E$ be a finite graph. Then every primitive ideal $P$ of $L=L_{K}(E)$ is
strongly primitive.
\end{proposition}

\begin{proof}
Let $P$ be a primitive ideal of $L$. In view of Remark \ref{Str. primitive},
we may assume that $P$ is a graded primitive ideal of the form $P=I(H,B_{H})$.
In fact, it is enough if we assume that $P=I(H,B_{H})$ is a prime ideal such
that $E\backslash(H,B_{H})$ satisfies Condition (L). We need to show that
$(E\backslash(H,B_{H}))^{0}=E^{0}\backslash H$ satisfies the strong CSP. If
$E\backslash(H,B_{H})$ contains a sink $w$, then, by the downward directness
of $(E\backslash(H,B_{H}))^{0}$, $u\geq w$ for every $u\in(E\backslash
(H,B_{H}))^{0}$ and this implies that $(E\backslash(H,B_{H}))^{0}$ satisfies
the strong CSP with respect to $\{w\}$. Suppose $(E\backslash(H,B_{H}))^{0}$
does not contain a sink. Since $E\backslash(H,B_{H})$ satisfies Condition(L)
and since $E\backslash(H,B_{H})$ is finite, it follows from Lemma
\ref{Lemma 6} that every vertex in $(E\backslash(H,B_{H}))^{0}$ connects to an
extreme cycle. Observe that, by downward directness of the graph, any two
extreme cycles $c,d$ in $E\backslash(H,B_{H})$ are connected. Consequently, if
$c$ is an extreme cycle in $E\backslash(H,B_{H})$ based at a vertex $v$, then
$u\geq v$ for every $u\in(E\backslash(H,B_{H}))^{0}$ . It is then clear that
$(E\backslash(H,B_{H}))^{0}$ satisfies the strong CSP with respect to $\{v\}$.
This shows that $P$ is strongly primitive.
\end{proof}

The next Proposition shows that a locally closed prime ideal is rational. In
the proof, we use the following Lemma.

\begin{lemma}
\label{Prop5.1 in ABR} (Proposition 5.1, \cite{ABR-1}) Let $K$ be a field and
let $A$ be a primitive $K$-algebra. If $\dim_{K}(A)<|K|$, then\ the extended
centroid $C(A)$ is algebraic over $K$.
\end{lemma}

\begin{proposition}
\label{Str. Prim => Rational} Let $E$ be an arbitrary graph and let $P$ be a
prime ideal of $L=L_{K}(E)$. If $P$ is locally closed, then $P$ is rational.
\end{proposition}

\begin{proof}
The proof is a minor modification of the one given in the proof of (1) =
$>$
(2) of Theorem 1.2 in \cite{ABR-1}. Suppose that $P$ is locally closed. By
Theorem \ref{loc. closed <=>Strong. primitive}, $P$ is strongly primitive and,
in particular, primitive. Let $\ K^{\prime}$ be an uncountable purely
transcendental field extension of $K$ having transcendence degree
$>\max\{|K|,|E|\}$. Note that $L_{K^{\prime}}(E)=L_{K}(E)\otimes_{K}K^{\prime
}$ and $P^{\prime}=P\otimes_{K}K^{\prime}$ is a primitive ideal of
$L_{K^{\prime}}(E)$. Let $R=L_{K}(E)/P$ and $R^{\prime}=L_{K^{\prime}%
}(E)/P^{\prime}$. We now have an injective homomorphism $\phi:C(R)\otimes
_{K}K^{\prime}\longrightarrow C(R^{\prime})$ defined as follows: If $I$ is a
nonzero ideal of $R$ and $f:I\rightarrow R$ is a bimodule homomorphism, then,
if $[(I,f)]\in C(R)$ is the equivalence class associated to $f$, we define
$\phi([I,f)]$ to be the class $[I^{\prime}.f^{\prime}]$ where $f^{\prime
}=f\otimes id:I\otimes_{K}K^{\prime}\longrightarrow R\otimes_{K}K^{\prime
}=R^{\prime}$. Since $R^{\prime}$ is primitive and $\dim_{K^{\prime}}%
R^{\prime}<|K^{\prime}|$ (Note that $\dim_{K}L_{K}(E))\leq|E|\aleph_{0}$ for
any field $K$), we conclude that $C(R^{\prime})$ is algebraic over $K^{\prime
}$, by Proposition \ref{Prop5.1 in ABR}. Given $[I,f]\in C(R)$, we see that
the bimodule homomorphism $[I^{\prime},f^{\prime}]$ is algebraic over
$K^{\prime}$. Thus there exists some non-zero ideal $J^{\prime}\subseteqq
I^{\prime}$ and some natural number $n$ such that
\[
\{(f\otimes id)^{n}|_{J^{\prime}},\cdot\cdot\cdot,(f\otimes id)|_{J^{\prime}%
},id|_{J^{\prime}}\}
\]
is linearly dependent over $K^{\prime}$. Thus $\{f^{n}|_{J},\cdot\cdot
\cdot,f|_{J},id|_{J}\}$ is linearly dependent over $K$ where $J=J^{\prime}\cap
R$. Hence $[I,f]$ is algebraic over $K$. Consequently, $C(R)$ is algebraic
over $K$, thus showing $P$ is a rational prime ideal.
\end{proof}

But, in general, a rational prime ideal need not be locally closed as the
following\ two examples show. Our first example is just the Example
\ref{Primitive need not be str. primitive} that we considered earlier. The
second example is due to Jason Bell. I am grateful to him for providing such a
nice example. In justifying the conclusions, we shall be using the following
important Proposition from \cite{AB}.

\begin{proposition}
\label{Rational prime} (Corollary 2.3, \cite{AB}) Suppose $R$ is von Neumann
regular prime algebra over a field $K$ with $\dim_{K}(R)<|K|$. Then the
extended centroid $C(R)$ of $R$ is an algebraic extension of the field $K$ and
thus $(0)$ is a rational prime ideal of $R$.
\end{proposition}

\begin{example}
Consider the graph $E$ of Example \ref{Primitive need not be str. primitive}.
Let $K$ be any uncountable field. Since $E$ acyclic, $L_{K}(E)$ is von Neumann
regular algebra (\cite{AR}) and since $E^{0}$ is downward directed, $L_{K}(E)$
is also a prime algebra. Thus, $\{0\}$ is a prime ideal and, as $E$ is a
countable graph, $\aleph_{0}=\dim_{K}(L_{K}(E))<|K|$. Then Proposition
\ref{Rational prime} implies that $\{0\}$ is a rational prime ideal of
$L_{K}(E)$. But $\{0\}$ is not locally closed, since $\{0\}=%
{\displaystyle\bigcap}
\{P_{n}:n>0\}$, where, for each $n>0$, $P_{n}$ is the ideal generated by the
vertex $(n,0)$ and is a prime ideal as $E^{0}\backslash(P_{n}\cap E^{0})$ is
downward directed.
\end{example}

\begin{example}
\label{rational is not locally closed} Let $E$ be a graph with\ its vertex set
$E^{0}=\{(i,j):i,j\in%
\mathbb{Q}
$, the set of rational numbers$\}$. Given $(i,j),(r,s)\in E^{0}$, there is an
edge $(i,j)\longrightarrow(r,s)$ whenever $i>r$ or $i=r$ and $j>s$.\ It is
then clear that every vertex in $E$ is an infinite emitter. Let $L:=L_{K}(E)$
where $K$ is an uncountable infinite field. For each rational number $i$, let
$H_{i}=\{(x,y)\in E^{0}:x\leq i\}$. It is clear that the set $H_{i}$ is
hereditary. It is also vacuously saturated, as every vertex in $E$ is an
infinite emitter. Also, for each $i\in%
\mathbb{Q}
$, $E^{0}\backslash H_{i}$ is downward directed. To see this, suppose
$(a,b),(c,d)\in E^{0}\backslash H_{i}$. Clearly, $a>i$ and $c>i$. Without loss
of generality, assume $a\leq c$. Then for some rational number $r<b,d$ we
have\ a vertex $(a,r)\in E^{0}\backslash H_{i}$ and edges
$(a,b)\longrightarrow(a,r)$ and $(c,d)\longrightarrow(a,r)$. This shows that
$E^{0}\backslash H_{i}$ is downward directed. Consequently, for each $i\in%
\mathbb{Q}
$, the ideal $P_{i}=<H_{i}>$ is a graded prime ideal of $L$. Moreover, it is
easy to see that $P_{i}=%
{\displaystyle\bigcap\limits_{i<j\in\mathbb{Q}}}
P_{j}$. Thus the prime ideal $P_{i}$ is not locally closed. We claim $P_{i}$
is a rational prime ideal. To see this, first note that the graph $E$ is
acyclic and so $L$ is von Neumann regular (\cite{AR}). Thus $L/P$ is a prime
von Neumann $K$-algebra and is countable dimensional since the graph $E$ has
countable number of vertices and countable number of edges. Since the field
$K$ has uncountable cardinality, we appeal to Proposition \ref{Rational prime}
to conclude that $P$ is a rational prime ideal.
\end{example}

As a corollary to the preceding analysis, we thus are able to append
additional equivalent properties to the DM-equivalence of a prime ideal in a
Leavitt path algebra of a finite graph, thus extending  Theorem 1.2 of
\cite{ABR-1}.

\begin{corollary}
\label{Finite E} Let $E$ be a finite graph and $L=L_{K}(E)$. Then the
following properties are equivalent for any prime ideal $P$ of $L$:

(i) \ \ $P$ is primitive;

(ii) \ $P$ is locally closed;

(iii) $P$ is rational;

(iv) $P$ is strongly primitive;

(v) \ $P$ is completely irreducible;

(vi) There is a non-zero element $a\in L$ such that $P$ is maximal with
respect to the property that $a\notin P$.
\end{corollary}

\textbf{Acknowledgement:} I am thankful to Jason Bell for many helpful
suggestions and, in particular, for providing the nice Example
\ref{rational is not locally closed}. I also thank Zak Mesyan for suggesting
to include Corollary \ref{Finite E}.

\end{document}